\documentclass[12 pt]{amsart}

\usepackage{latexsym, amsmath, amssymb, longtable, booktabs,amscd}
\usepackage[square]{natbib}
\usepackage[english]{babel}
\usepackage[utf8x]{inputenc}
\usepackage{hyperref}
\usepackage{comment}
\bibpunct{[}{]}{,}{n}{}{;}
\usepackage{etoolbox}

\usepackage[colorinlistoftodos]{todonotes}
\usepackage{graphicx}
\usepackage{cleveref}
\numberwithin{equation}{section}
\usepackage{tikz-cd}

\newcommand\N{\mathbb{N}}

\newcommand\sO{\mathcal{O}}

\newcommand\fP{\mathfrak{P}}

\newcommand\cS{\mathcal{S}}
\newcommand{\cC}{\mathcal{C}}
\newcommand\fp{\mathfrak{p}}

\newcommand{\Q}{\mathbb{Q}}

\newcommand{\I}{\mathbb{I}}
\newcommand{\Z}{\mathbb{Z}}

\newcommand\Gal{{\mathrm {Gal}}}

\DeclareMathOperator{\lcm}{lcm}
\DeclareMathOperator{\rank}{rank}
\newcommand\remove[1]{}
\usepackage{amsthm}
\newtheorem{theorem}{Theorem}[section]

\newtheorem{lemma}[theorem]{Lemma}

\theoremstyle{definition}
\newtheorem{definition}[theorem]{Definition}

\theoremstyle{remark}
\newtheorem{remark}[theorem]{Remark}

\author{ Srilakshmi Krishnamoorthy}
\author{ Sunil Kumar Pasupulati }

\address{Srilakshmi Krishnamoorthy
	\newline
	\ \ INDIAN INSTITUTE OF SCIENCE EDUCATION AND RESEARCH, TRIVANDRUM, INDIA.}
\email{srilakshmi@iisertvm.ac.in}
\address{ Sunil Kumar Pasupulati
	\newline
	\ \ INDIAN INSTITUTE OF SCIENCE EDUCATION AND RESEARCH, TRIVANDRUM, INDIA.}
\email{sunil4960016@iisertvm.ac.in}

\title[	Euclidean Ideal Class in Biquadratic Fields]{
Non-principal Euclidean ideal class in a family of biquadratic fields with  the class number two}

\begin{document}
	
	\begin{abstract}
		Lenstra introduced the notion of a Euclidean ideal class, which is a generalization of the Euclidean domain. Lenstra also proved that the Euclidean ideal in a number field $K$ implies that the class group of $K$ is cyclic. We construct a family of biquadratic fields that have a  Euclidean ideal whenever the class number is 2. This extends the families given by Graves, Hsu,  Chattopadhyay and  Muthukrishnan.  
	\end{abstract}
	
	\maketitle
	\section{Introduction}
	For a number field $K$, let    $\sO_K,Cl_{K},h_K, \sO_K^{\times},H(K)$, and $f(K)$  denote the  ring of integers of $K$, the ideal class group of  $K$,   
	the class number of $K$,  the multiplicative group of units  of $K$,  the Hilbert class field of $K$, and the  conductor of $K$, respectively.  If $L/K$ is an  extension of number fields and a prime $\fP$ of $L$  lies above a prime $\fp$ of $K$, then   the   residue degree $f(\fP/\fp)$   is equal to  the  dimension of $\sO_L/\fP $ as $\sO_K/\fp$ vector space.  
	
	In 1979, Lenstra \cite{len} introduced  Euclidean ideals, which are  generalization of Euclidean domains.
\begin{definition}
	Let $R$ be a Dedekind domain and $\I$ be the set of all non-zero integral ideals of $R$. The ideal  $C \in \I$   is called a  {\it Euclidean ideal} if there exists a function $\Psi:\I \to W$, where $W$ is  a well-ordered set, such that for all $I \in \I $ and   every $x\in I^{-1}C \setminus C$,  there exists  a  $y\in C$ such that $$ \Psi\left( (x-y)IC^{-1} \right) < \Psi (I).$$
	We say  $\Psi$ is  a Euclidean function for $C$.
	If $C$ is a Euclidean ideal, then every ideal in the ideal class $[C]$ is also a Euclidean ideal, and the ideal class $[C]$ is called a {\it Euclidean ideal class.}
\end{definition}
Lenstra \cite{len}  proved that if  $C$ is a Euclidean ideal, then  the ideal class  $[C]$ generates $Cl_K$,  hence   $Cl_K$ is cyclic. Nevertheless the converse need not be true. For example, the  imaginary quadratic fields  $\Q(\sqrt{-d})$,  where  $d= 19, 43, 67, 163,$ have trivial class groups but do not contain a Euclidean ideal.
Lenstra \cite{len} proved that suppose $ K$ is a number field with  $\rank(\sO_K^{\times})\geq 1$,  under the assumption of  the generalized Riemann hypothesis
(GRH),  $Cl_K$ is cyclic if and only if $ K$ has a Euclidean ideal class. 
More precisely, assuming the GRH, a number field, not an imaginary quadratic field, has a cyclic class group if and only if it has a Euclidean ideal class.

		Graves \cite{G1} constructed an explicit biquadratic field $\Q(\sqrt{2},\sqrt{35})$ which has a non-principal Euclidean ideal class by using a growth result (\Cref{G1}). This was generalized by Hsu 
	\cite{CH}, in which the author  explicitly constructed a family of  real biquadratic fields $K= \Q(\sqrt{q},\sqrt{kr})$  that have a non-principal Euclidean ideal class, whenever $h_K =2$ and  $q,k,r$  are all primes congruent to 1 modulo 4 which are greater than $29$.
	Chattopadhyay and Muthukrishnan \citep{JS19}  extended Hsu's family of biquadratic fields by removing the  condition   $q \equiv 1\mod4$.

	In this paper, we provide the following new family of real biquadratic fields $K$ that have a non-principal Euclidean ideal class whenever  $h_K=2$ by assuming only 
	one of $r,k$ is congruent to $1$ modulo $4$.
	\begin{theorem}{\label{main}}
		Let $K=\Q(\sqrt{q},\sqrt{kr})$,   either  $k$ or $r$
		is congruent to $1 $ modulo $4$.  Suppose that  $q,r$, and $k$  are prime numbers, none of  which is $ 2,3,5,7$  or $17$ and $h_K= 2$, then $K$ has a non-principal Euclidean ideal class.
	\end{theorem}
	
	The real biquadratic field $K =\Q(\sqrt{q},
	\sqrt{kr})$  has  unit rank $3$  but  we cannot apply \cite[Theorem 2] {DGS}  and  conclude that K has a Euclidean ideal, as 
	 $\Gal\left(\Q\left(\zeta_{f\left(H(K)\right)}\right)/K\right)$ is not necessarily  cyclic, however we may apply  \Cref{main}. For example,  $K=\Q(\sqrt{13},\sqrt{19 \times37})$, then  some simple SageMath calculations and  \Cref{lm31} imply  that the Hilbert class field of $K$ is   $\Q(\sqrt{13},\sqrt {19},\sqrt{37}) $  and its  conductor is equal to 36556. The Galois group 
		$\Gal(\Q(\zeta_f)/K) \cong \Z/12\Z\times \Z/9\Z\times \Z/36\Z$  is not cyclic. But we can conclude that $K$ has a non-principal Euclidean ideal class by \Cref{main}.

The organization of the paper is as follows. In \Cref{sec2}, we write some useful preliminaries for proving the main result. In \Cref{sec3}, we prove \Cref{main}. In \Cref{sec4},  we provide some examples of real biquadratic fields satisfying the assumptions of  \Cref{main}.
	\section{Preliminaries}\label{sec2}
	\subsection{Definitions and results}
	\begin{definition}
		Let $K/\Q$ be a finite abelian extension. The {\it conductor} of  $K$, which is denoted by $f(K)$ is the smallest even number such that  $K \subset \Q(\zeta_{f(K)})$.
	\end{definition}
	\begin{definition}
		The {\it Hilbert class field } $H(K)$ of  a number field $K$ is  the maximal unramified abelian extension of $K$.
	\end{definition}
	For example,  If $K=\Q(\sqrt{11},\sqrt{43\times 13})$,  then $H(K) = \Q(\sqrt{11},\sqrt{43}, \sqrt{13})$.
	\begin{definition}
Let  $K$ be a number field and $\fp$ be a prime  of $K$ lies over  a rational prime $p$. The prime $\fp$ is called a prime of first degree if $\sO_K/\fp\cong\Z/p\Z$.
	\end{definition}
	\begin{lemma}\label{L1}
	Let $L_1,L_2$ be finite extensions of number field $K$. Suppose the  prime $p$ in $K$ is unramified in $L_1$. If $\fP_2$  is  a prime of $L_2$  lying above $p$,  then  $\fP_2$ is unramified in $L_1L_2 $ over $L_2$.
\end{lemma}
\begin{proof}
	The proof is evident from   \cite[Lemma 2.2]{KNS} 		
\end{proof}

	We now state two theorems which provide the upper bounds for the least prime quadratic residue  and non-residue that are congruent to $3$  modulo $4$.
	\begin{theorem}\label{Q1}(Gica \cite{AG06})
		If $p$ is a prime not equal to $2,3,5,7,$ and $17,$ then there exists a prime $q<p$ and $q\equiv 3 \pmod{4}$  which is a quadratic residue modulo $p$.
	\end{theorem}
	\begin{theorem}\label{Q2}(Pollack \cite{PP18} )
		For every prime $p \geq 5$, there is a prime $q<p$ and $q\equiv 3\pmod{4}$  which is a quadratic non residue  modulo $p$. 
	\end{theorem}

\subsection{Few results on Euclidean ideals }
	Graves \cite{growth}  proved a useful growth result,  which gives a condition for the existence of a Euclidean ideal in  number fields, without the assumption of the  GRH. We state the result below, which is one of the main ingredients for proving our theorem.
	\begin{theorem}{\label{G1}}(Graves \cite{growth})
		Suppose that K is a number field such that $\rank(\sO_K^{\times})\geq 1$
		and that $C$ is a non-zero ideal of $\sO_K$. If $[C]$  generates the class group of $K$ and
		\[\Biggr| \{  \fP \subset \sO_K:   \fP  \text{\  is \ prime }, \ \N(\fP)\leq x, \ [\fP]=[C], \ \pi_\fP \text{\ is \ onto} \}\Biggr|\gg \frac{x}{(\log x)^2},
		\]
		where $\pi_\fP$ is the canonical map from $\sO_K^{\times} \to (\sO_K/\fP)^{\times}$,  then $[C]$ is a Euclidean ideal class.
	\end{theorem}
	   Using the above growth result,  Graves and Murty  \cite{GM}  proved the  existence of a  Euclidean ideal class in Galois number field $K$   with abelian  Hilbert class field $H(K)$ and $\rank(\sO_K^{\times})\geq 4$.
	 
\begin{theorem} (Graves \cite{G1})\label{G2}
	Let K be a totally real number field with conductor $f(K)$ and let
	$\{ e_1,e_2,e_3\}$ be a multiplicatively independent set contained
	in $\sO_K^{\times}$. If $l= \lcm (16,f(K)),$ and if
	$\gcd (u,l)=\gcd(\frac{u-1}{2},l)=1$ for some integer u, then
	
	\begin{align*}
		\left|  \begin{Bmatrix}
			\fP \subset \sO_K,\quad \fP \ 	\text{is a prime  }    &\Bigm\vert &N(\fP)\equiv u \pmod{l}, \  N(\fP)\le x,   \\
			\text{ of  first degree  }  &\Bigm\vert & \langle-1,e_i\rangle\twoheadrightarrow (\sO_K/\fP)^{\times}
		\end{Bmatrix}
		\right|\gg \frac{x}{(\log x)^2},
	\end{align*}
	for at least one $i$, where $ \langle-1,e_i\rangle$ is multiplicative group generated by  $-1$  and  $e_i$.
\end{theorem}

\subsection{The Artin symbol and its applications}
	Let us state a lemma, which is essential to define the Artin symbol.
	\begin{lemma} \label{al}
		Let $K\subset L$ be a Galois extension of number fields and  let $\fp$ be a prime
		in  $\sO_K$ which is unramified in L. If  $\fP$ is prime of $\sO_L$  containing
		$\fp$, then there exists  a unique element $\sigma \in Gal( L/K)$ such that for
		all $\alpha \in \sO_L$, 
		 \begin{align*}
		\sigma(\alpha)\equiv\alpha^{Nm(\fp)} \pmod {\fP},
		\end{align*}
		where $Nm(\fp) =|\frac{\sO_K}{\fp}|$ is the norm of $\fp$.
	\end{lemma}
	\begin{proof}
		Refer to 	\cite[Lemma 5.19]{cox}.
	\end{proof}
	\begin{definition}
		The unique element $\sigma$ in  the above lemma is called  the  {\it Artin symbol} and it is denoted by $\left( \frac{L/K}{\fP}\right)$.
	\end{definition}
	In the next  lemma, we state some useful properties of  the Artin symbol which will be used in later part of  this paper.
	\begin{lemma}
		Let $K\subset L$ be a Galois extension and  the prime  $\fp\subset \sO_K$ be unramified in $L$.  Given a prime $\fP$ of $L$ containing $\fp$, 
		\begin{enumerate}
			\item If $\sigma \in Gal(L/K)$, then  $$ \left(\frac{L/K}{\sigma(\fP)}\right)=\sigma\left(\frac{L/K}{\fP}\right)\sigma^{-1}.$$
			\item  The prime $\fp$ splits completely  in $L$ if and only if $\left(\frac{L/K}{\fP}\right)=1$.
		\end{enumerate}
	\end{lemma}
	\begin{proof}
		Refer to 	\cite[Lemma 5.21]{cox}.
	\end{proof}
	\begin{remark}
		Let  $L/K$ be  an abelian extension and the prime $\fp$  of $K$ is unramified in $L$. For any two primes $\fP_1,\fP_2 $ of $L$ which lie above a prime $\fp$, we have  $\left(\frac{L/K}{\fP_1}\right)=\left(\frac{L/K}{\fP_2}\right)$. Therefore for any prime $\fP$ in $L$, which lies above $\fp$,  the Artin symbol is also denoted by $\left(\frac{L/K}{\fp}\right)$.	
	\end{remark}
	\begin{definition}
		Let $K$ be a number field and  let $\mathbb{P}_K$ be  the set of all prime ideals of $\sO_K$. If  $\cS$ is a subset of $\mathbb{P}_K$,  then  the  {\it Dirichlet density} of the set $\cS$  is defined to be 
		$$ \delta(\cS)= \text{lim}_{s\to 1^+}\frac{\sum_{\fp\in\cS }Nm(\fp)^{-s}}{-\log(s-1)},$$
		provided  the limit exists, where $Nm(\fp)=|\sO_K/\fp|$.
	\end{definition}
	Now we recall the Chebotarev density theorem   \cite[Theorem 8.17]{cox}, which helps to calculate the Dirichlet density of a set.
	\begin{theorem}\label{density}
		Let $L/K$ be a finite Galois extension of global fields and  let $\cC$ be a conjugacy class in $\Gal(L/K)$. The Dirichlet density of the set  $$ \Big\{ \fp\in \mathbb{P}_K \Bigm\vert\  \fp \ \text{ is unramified in $L $ and }  \left( \frac{L/K}{\fp} \right)= \cC \Big\}$$  exists and  it is equal to $\frac{|\cC| }{[L:K]}$.
	\end{theorem}
	
	We end this section by stating a lemma  whose proof is evident from \cite[Corollary 5.21, Corollary 5.25]{cox}.

	\begin{lemma}\label{princi}
		Let $K $ be a number field. Suppose  $\fP$ and $\fp$ are primes of $H(K) $  and $K$ respectively,  which lies above  the  rational prime p. If $\fP/p$ has residue degree greater than or equal to $2$ and $\fp/p$ has residue degree $1$,  then  $\fp$  is a non-principal ideal.
	\end{lemma}
	\section{Proof of the Main Theorem}\label{sec3}
	We start this section by explicitly calculating the Hilbert class field of $\Q(\sqrt{q},\sqrt{kr}),$ which plays a crucial role in the proof of our main theorem. 
	\begin{lemma}\label{lm31}
		If $K=\Q(\sqrt{q},\sqrt{kr})$ and at least one of    $k,r$
		is congruent to $1 $ modulo $4$, then  $H(K)=\Q(\sqrt{q},\sqrt{k},\sqrt{r})$.
		
	\end{lemma}
	\begin{proof}
		Suppose $L=\Q(\sqrt{q},\sqrt{k},\sqrt{r})$, then the conductor    $f(L)= 4qkr$. Therefore  only primes that can ramify in $L/K$  are those primes that lie above $ 2,q,k, $ and $r$. 
		We have  at least one of $k,r$ is congruent to $1$ modulo ${4}$,  so without loss of generality, we assume that  $r\equiv 1 \pmod 4$.
		Since  $r\equiv 1 \pmod 4$, the prime  $2$ does not  ramify in $\Q(\sqrt{r})$.  Thus  by \Cref{L1},   the primes which lie above $2$ in  $K$ are unramified in $L$. 
		If  $p$ is one of $k,q,r,$ then  $p$ is unramified in  at least
		one of the  quadratic fields
		$\Q(\sqrt{r})$, $\Q(\sqrt{k})$. By  \Cref{L1},  the
		primes which lie above $p$ in  $K$ are unramified in $L$. Therefore all the primes of $K$  are unramified in $L$. As a consequence of the fact that the  Hilbert class field $H(K)$ of $K$ is the maximal abelian unramified extension of $K$,  thus $L\subset H(K)$.  Moreover, since   $Gal(H(K)/K)\cong Cl_K$, we obtain  $H(K)=L$.
	\end{proof}

	\begin{lemma}\label{inf}
		The primes which  split completely in $K=\Q(\sqrt{q},\sqrt{kr})$  but not in its  Hilbert class field $H(K)=\Q(\sqrt{q},\sqrt{k},\sqrt{r})$ have density $\frac{1}{8}$.
	\end{lemma}
	\begin{proof}
		Let $X_{K},X_{H(K)}$  denote the  sets of prime numbers  which split completely in $K$ and $H(K)$ respectively. Then 
		\begin{align*}
			X_K&=\left \{ p\ \Bigm\vert \ \text{ $p$  is prime and} \ \left(\frac{K/\Q}{p}\right)=1\right\},\\
			X_{H(K)}&= \left \{ p \ \Bigm\vert \ \text{ $p$  is prime and} \ \left(\frac{H(K)/\Q}{p}\right)=1\right\}.
		\end{align*}
		Since $K\subset H(K)$, it follows that $X_{H(K)} \subset X_{K}$.  
		The identity elements  in $\Gal(K/\Q)$, and $\Gal(H(K)/K)$ each form their own  conjugacy class. By  \Cref{density}, the Dirichlet densities of $X_K$ and $X_{H(K)}$ are 
		$\frac{1}{4}$ and $\frac{1}{8}$  respectively. Since $X_{H(K)} \subset X_{K},$  we conclude  that the  Dirichlet
		density of the set $X_ {K} \setminus  X_{H(K)}$ is $\frac{1}{4}-\frac{1}{8}=\frac{1}{8}$,  therefore $X_ {K} \setminus  X_{H(K)}$ is infinite.	
	\end{proof}
\subsection*{Proof of Theorem \ref{main}}  
	Since  $K =\Q(\sqrt{q},\sqrt{kr})$ is  a totally real number field of degree $4$,    $\rank(\sO_K^{\times}) =3$, and  there exists a   multiplicatively independent set    $\{e_1,e_2,e_3\}$     in $\sO_K^{\times}$. Let $l=16qkr$. 
	Next we will  find an integer $ u $  satisfying  the following conditions : 
	\begin{enumerate}
		\item $(u,l)=1,$ \label{c1}
		\item $(\frac{u-1}{2},l)=1,$ \label{c2} 
		\item If a prime  $p$ satisfies $p\equiv u \pmod{4qkr}$, then $f(\fp/p)=1 $  and  $f(\fP/p)=2$. 
		\label{c3} 
	\end{enumerate}
	By \Cref{inf},  there are infinitely many primes satisfying condition \ref{c3}. 
	Observe that a prime $p$ splits completely in $K=\Q(\sqrt{q},\sqrt{kr})$ if and only if $p$ splits completely in  $\Q(\sqrt{q})$ and $\Q(\sqrt{kr})$. This  translates to 
	\begin{align*}
		\left( \frac{q}{p}\right)=1 \ \text{and} \ \left( \frac{kr}{p}\right)=1.
	\end{align*}
	Similarly,  a prime $p$ splits completely in  $H(K) =\Q(\sqrt{q},\sqrt{k},\sqrt{r})$ if and only if
\begin{align*}
	\left( \frac{q}{p}\right)=1, \left( \frac{k}{p}\right)=1, \ \text{and} \ \left( \frac{r}{p}\right)=1.
\end{align*}
	Thus a prime   $p \notin X_{H(K)}$  if and only if  one  among $ \left( \frac{q}{p}\right),  \left( \frac{k}{p}\right), \left( \frac{r}{p}\right) $ is equal to $-1$. 
	Hence
	\begin{align*}{\label{e1}}
		X_K 	\setminus X_{H(K)} = \Biggr \{ p: \ \text{ $p$  is prime and} \ \left(\frac{q}{p}\right)=1 \ \text{and} \ \left( \frac{k}{p}\right)=\left( \frac{r}{p}\right)=-1\Biggr\}.
	\end{align*}
	
We demonstrate how $u$ can be found by using the Chinese remainder theorem. 
 For that, we translate the conditions(1)-(3) into simultaneous congruences.
 
Since  either  $k$ or $r$ is congruent to $1$  modulo $4$, without loss of generality, we can assume that $k\equiv 3\pmod{4},r\equiv 1\pmod{4}$.
	By \Cref{Q1} and \Cref{Q2},  there exist prime numbers $p_1<q,p_2<k,p_3<r$ that  are congruent   to $3$ modulo $4$ and satisfy
	$$ \left( \frac{p_1}{q}\right) =(-1)^{\frac{q-1}{2}}, \left(\frac{p_2}{k}\right)=1 \ \text{and} \ \left(\frac{p_3}{r}\right)=-1. $$
	Now consider the following system of simultaneous congruences.
	\begin{equation}\label{e1}
	x\equiv p_1 \pmod{q}, 
\end{equation}
\begin{equation}\label{e2}
	x\equiv p_2 \pmod{k},
\end{equation}	
\begin{equation}\label{e3}
	x\equiv p_3 \pmod{r},
\end{equation}	
\begin{equation}\label{e4}
	x\equiv 3 \pmod{4}.
\end{equation}
	The Chinese remainder theorem promises a unique solution modulo $4qkr$ to the  above system of congruences, say $x_0$.
	Since $\gcd(x_0,4qkr)=1$,  by  Dirichlet’s theorem on primes in arithmetic progressions, there exist infinitely many prime numbers $p$ satisfying $p\equiv x_0\pmod{4qkr}$. We pick one  such  prime and call it $u$. It is easy to observe that $u$ satisfies  condition (\ref{c1}).  If $\gcd(\frac{u-1}{2},l) \neq 1$, then  either there exists an odd prime $p \mid l$   such that $u\equiv 1\pmod p$  or  $u\equiv 1\pmod 4$. This contradicts that   $u$  is  a solution of the above simultaneous congruences (\ref{e1})-(\ref{e4}). Therefore $u$ satisfies  condition (\ref{c2}). 
	
	Now by  quadratic reciprocity,  we have 
	\begin{align*}
		\left(\frac{q}{u}\right)&= (-1)^{\frac{u-1}{2}\frac{q-1}{2}}\left(\frac{u}{q}\right) =(-1)^{\frac{q-1}{2}}\left(\frac{p_1}{q}\right)=(-1)^{\frac{q-1}{2}}(-1)^{\frac{q-1}{2}}=1,\\	
		\left(\frac{k}{u}\right)&= (-1)^{\frac{k-1}{2}\frac{u-1}{2}}\left(\frac{u}{k}\right) =(-1)\left(\frac{p_2}{k}\right)=-1,\\
				\left(\frac{r}{u}\right)&= (-1)^{\frac{r-1}{2}\frac{u-1}{2}}\left(\frac{u}{r}\right) =\left(\frac{p_3}{r}\right)=-1.\end{align*}
	Hence   $u \in X_{K}\setminus X_{H(K)}$  and  also satisfies condition (\ref{c3}).

	Let  $u$ satisfy  conditions   (\ref{c1})-(\ref{c3}).  If $p \equiv u \pmod{4qkr}$, then every prime  $\fp \subset \sO_K$ which lies above $p$ does not
	split in $H(K)$ and   $ f(\fP/\fp)=2 $, where $\fP$ is  any prime of $H(K)$ which lies above $\fp$. Hence by \Cref{princi},
	$\fp$ is not a principal ideal, and  since  $h_K=2$,  $[\fp]$ is  the  generator of $Cl_K$. For each $i\in \{1,2,3\},$ 
		\[\left|  \begin{Bmatrix}
		&\Bigm\vert & N(\fP)\leq x, \\
		\fP\subset \sO_K &\Bigm\vert  & [\fP]=[C], \\
		\text{prime }	& \Bigm\vert&\pi_p \text{\ is \ onto}
	\end{Bmatrix}
	\right| >  \left|  \begin{Bmatrix}
		\text {prime } &\Bigm\vert& N(\fP)\equiv u \pmod{l}, \\
		\text { ideals of  }  &\Bigm\vert & N(\fP)\le x, \\
		\text{ first degree} &\Bigm\vert& \langle-1,e_i\rangle\twoheadrightarrow (\sO_K/\fP)^{\times}
	\end{Bmatrix}
	\right|. \]
Now using  	\Cref{G2},  we obtain
	\[\left|  \begin{Bmatrix}
	 \fP\subset \sO_K \ \text{prime ideal }&\Bigm\vert & N(\fP)\leq x,[\fP]=[C],\pi_p \text{\ is \ onto} \\
	\end{Bmatrix}
	\right|   \gg \frac{x}{(\log x)^2}.  \]
	
	Applying \Cref{G1}, we conclude that $K$ has a non-principal  Euclidean ideal, i.e., $K$ has a nontrivial Euclidean ideal class. 
	\section{Examples}\label{sec4}
	There are plenty of biquadratic fields $\Q(\sqrt{q},\sqrt{kr})$ with class
	number two. We present few examples of $\Q(\sqrt{q},\sqrt{kr})$ which
	satisfy  the hypothesis of \Cref{main}. All the computations are  done in  SageMath \nocite{sage}  and it can be found at \href{https://github.com/sunilpasupulati/existence-of-a-non-principal-Euclidean-ideal-class-in-biquadratic-fields} {https://github.com/sunilpasupulati}
\href{https://github.com/sunilpasupulati/existence-of-a-non-principal-Euclidean-ideal-class-in-biquadratic-fields} {/existence-of-a-non-principal-Euclidean-ideal-class-in-biquadratic-fields}.\\
	\begin{tabular}{ p{4cm} p{4cm}p{3cm}p{3cm}}
		\hline  \hline
		$(q,k,r)$ & $h_{\Q(\sqrt{q},\sqrt{kr})}$ & $(q,k,r)$ & $h_{\Q(\sqrt{q},\sqrt{kr})}$    \\
		\hline
		\hline
		(11,19,13)&2&(11,67,13) &2 \\
		(11,23,13)&2&(11,71,13) &2 \\
		(11,31,13)&2&(19,11,13) &2 \\
		(11,43,13)&2&(19,23,13) &2 \\
		(11,47,13)&2&(19,31,13) &2 \\
		(11,59,13)&2&(19,43,13) &2\\
		\hline
	\end{tabular}
	\begin{center}
		Table 1: Examples of $\Q(\sqrt{q},\sqrt{kr})$ with class number 2 and $q\equiv 3, k\equiv 3, r\equiv 1 \pmod{4}$.
	\end{center}
	\vspace{0.75cm}
	\begin{tabular}{ p{4cm} p{4cm}p{3cm}p{3cm}}
		\hline  \hline
		$(q,k,r)$ & $h_{\Q(\sqrt{q},\sqrt{kr})}$ & $(q,k,r)$ & $h_{\Q(\sqrt{q},\sqrt{kr})}$    \\
		\hline
		\hline
		(13,11,37) &2  &(37,11,29) &2 \\
		(13,11,41) &2 &(37,11,61) &2 \\
		(29,11,37) &2 &(37,11,109) &2 \\
		(29,11,41) &2& (41,11,13)&2\\
		(29,11,73) &2& (41,11,29)&2\\
		(37,11,13) &2& (41,11,53)&2 \\
		\hline
	\end{tabular}
	\begin{center}
		Table 2: Examples of $\Q(\sqrt{q},\sqrt{kr})$ with class number 2 and $q\equiv 1, k\equiv 3, r\equiv 1 \pmod{4}$. 
	\end{center}
	\section{Concluding Remarks}
	
	The assumption  $q,r,k \notin \{2,3,5,7,17\}$ in \Cref{main}  is not a necessary  condition. By using different method, we\cite{SP03}  proved the existence of euclidean ideal  $K =\Q(\sqrt{q},\sqrt{kr})$ without any conditions on $q,k$, and $r$  whenever $h_K=2$. 
	
	Lenstra \cite{len} also defined norm Euclidean ideals, which are the  Euclidean ideals  where the ideal norm is a Euclidean function.
		If K is a number field and C is  a fractional ideal of $\sO_K$, then C is norm Euclidean if for all $x\in K $,  there exists $y\in C$ such that 
			$Nm(x-y)<Nm(C)$. 
		
	Lezowski \cite{LP12} showed that the biquadratic field $\Q(\sqrt{2},\sqrt{35})$  does not contain a norm Euclidean ideal by  using  the Euclidean minimum. 
	It will be interesting to  know whether the biquadratic fields $K = \Q(\sqrt{q},\sqrt{kr})$ (under the assumptions of \Cref{main}) have norm Euclidean ideals.

	\section*{Acknowledgments}
	We thank  Viji Z Thomas for valuable discussions. We also thank Patali and  Shanmugapriya for their careful reading of this manuscript. We thank B.sury  for introducing the problem. We acknowledge IISER Tiruvananthapuram for providing excellent working conditions. We want to express profound gratitude to the anonymous referee for carefully analyzing the article and making suggestions for a better presentation. 
	
	\bibliographystyle{amsplain}
	\bibliography{ref.bib}

\providecommand{\bysame}{\leavevmode\hbox to3em{\hrulefill}\thinspace}
\providecommand{\MR}{\relax\ifhmode\unskip\space\fi MR }
\providecommand{\MRhref}[2]{%
  \href{http://www.ams.org/mathscinet-getitem?mr=#1}{#2}
}
\providecommand{\href}[2]{#2}
\begin{thebibliography}{10}

\bibitem{JS19}
Jaitra Chattopadhyay and Subramani Muthukrishnan, \emph{Biquadratic fields
  having a non-principal {E}uclidean ideal class}, J. Number Theory
  \textbf{204} (2019), 99--112. \MR{3991414}

\bibitem{cox}
David~A. Cox, \emph{Primes of the form {$x^2 + ny^2$}}, A Wiley-Interscience
  Publication, John Wiley \& Sons, Inc., New York, 1989, Fermat, class field
  theory and complex multiplication. \MR{1028322}

\bibitem{DGS}
J.-M. Deshouillers, S.~Gun, and J.~Sivaraman, \emph{On {E}uclidean ideal
  classes in certain abelian extensions}, Math. Z. \textbf{296} (2020),
  no.~1-2, 847--859. \MR{4140766}

\bibitem{AG06}
Alexandru Gica, \emph{Quadratic residues of certain types}, Rocky Mountain J.
  Math. \textbf{36} (2006), no.~6, 1867--1871. \MR{2305634}

\bibitem{G1}
Hester Graves, \emph{{$\Bbb Q(\sqrt{2},\sqrt{35})$} has a non-principal
  {E}uclidean ideal}, Int. J. Number Theory \textbf{7} (2011), no.~8,
  2269--2271. \MR{2873154}

\bibitem{growth}
\bysame, \emph{Growth results and {E}uclidean ideals}, J. Number Theory
  \textbf{133} (2013), no.~8, 2756--2769. \MR{3045214}

\bibitem{GM}
Hester Graves and M.~Ram Murty, \emph{A family of number fields with unit rank
  at least $4$ that has {E}uclidean ideals}, Proc. Amer. Math. Soc.
  \textbf{141} (2013), no.~9, 2979--2990. \MR{3068950}

\bibitem{CH}
Catherine Hsu, \emph{Two classes of number fields with a non-principal
  {E}uclidean ideal}, Int. J. Number Theory \textbf{12} (2016), no.~4,
  1123--1136. \MR{3484301}

\bibitem{KNS}
Joachim K\"{o}nig, Danny Neftin, and Jack Sonn, \emph{Unramified extensions
  over low degree number fields}, J. Number Theory \textbf{212} (2020), 72--87.
  \MR{4080045}

\bibitem{SP03}
Srilakshmi Krishnamoorthy and Sunil~Kumar Pasupulati, \emph{On the existence of
  euclidean ideal class in quadratic, cubic and quartic extensions}, arXiv
  preprint arXiv:2110.00225 (2021).

\bibitem{len}
H.~W. Lenstra, Jr., \emph{Euclidean ideal classes}, Journ\'{e}es
  {A}rithm\'{e}tiques de {L}uminy ({C}olloq. {I}nternat. {CNRS}, {C}entre
  {U}niv. {L}uminy, {L}uminy, 1978), Ast\'{e}risque, vol.~61, Soc. Math.
  France, Paris, 1979, pp.~121--131. \MR{556669}

\bibitem{LP12}
Pierre Lezowski, \emph{Examples of norm-{E}uclidean ideal classes}, Int. J.
  Number Theory \textbf{8} (2012), no.~5, 1315--1333. \MR{2949204}

\bibitem{PP18}
Paul Pollack, \emph{The least prime quadratic nonresidue in a prescribed
  residue class $\mod 4$}, J. Number Theory \textbf{187} (2018), 403--414.
  \MR{3766918}

\bibitem{sage}
W.\thinspace{}A. Stein et~al., \emph{{S}age {M}athematics {S}oftware ({V}ersion
  9.6)}, The Sage Development Team, {\tt http://www.sagemath.org}.

\end{thebibliography}

\end{document}